\documentclass[11pt]{article}

\usepackage{latexsym,amssymb,graphics}

\begin{document}
\title{\bf Chordal Graphs are Fully Orientable}

\author{Hsin-Hao Lai\\
\normalsize Department of Mathematics\\
\normalsize National Kaohsiung Normal University \\
\normalsize Yanchao, Kaohsiung 824, Taiwan\\
\normalsize {\tt Email:hsinhaolai@nknucc.nknu.edu.tw}
\and
Ko-Wei Lih\thanks{The corresponding author}\\
\normalsize Institute of Mathematics\\
\normalsize Academia Sinica\\
\normalsize Nankang, Taipei 115, Taiwan\\
\normalsize {\tt Email:makwlih@sinica.edu.tw}
}
\date{\small \today}
\maketitle

\newtheorem{define}{Definition}
\newtheorem{proposition}[define]{Proposition}
\newtheorem{theorem}[define]{Theorem}
\newtheorem{lemma}[define]{Lemma}
\newtheorem{remark}[define]{Remark}
\newtheorem{corollary}[define]{Corollary}
\newtheorem{problem}[define]{Problem}
\newtheorem{conjecture}[define]{Conjecture}
%
%
\newenvironment{proof}{
\par
\noindent {\bf Proof.}\rm}%
{\mbox{}\hfill\rule{0.5em}{0.809em}\par}

\baselineskip=16pt
\parindent=0.5cm


\begin{abstract}
\noindent
Suppose that $D$ is an acyclic orientation of a graph $G$.
An arc of $D$ is called {\em dependent} if its reversal
creates a directed cycle. Let $d_{\min}(G)$ ($d_{\max}(G)$)
denote the minimum (maximum) of the number of dependent arcs
over all acyclic orientations of $G$. We call $G$ {\em fully
orientable} if $G$ has an acyclic orientation with exactly
$d$ dependent arcs for every $d$ satisfying $d_{\min}(G)
\leqslant d \leqslant d_{\max}(G)$. A graph $G$ is called
{\em chordal} if every cycle in $G$ of length at least four
has a chord. We show that all chordal graphs are fully
orientable.

\end{abstract}

%
%
\section{Introduction}
%
%

Let $G$ be a finite graph without multiple edges or 
loops.
We use $|G|$ and $\|G\|$ to denote the number of vertices and
the number of edges of $G$, respectively. An orientation $D$ of
a graph $G$ is obtained by assigning a fixed direction, either
$x \rightarrow y$ or $y \rightarrow x$, on every edge $xy$ of
$G$. The original undirected graph is called the {\em underlying}
graph of any such orientation.

An orientation $D$ is called {\em acyclic} if there does not
exist any directed cycle. A directed graph having no directed
cycle is commonly known as a {\em directed acyclic graph}, or
DAG for short. DAGs provide frequently used data structures in
computer science for encoding dependencies. An equivalent way
of describing a DAG is the existence of a particular type of
ordering of the vertices called a {\em topological ordering}.
A topological ordering of a directed graph $G$ is an ordering
of its vertices as $v_1, v_2, \ldots , v_{|G|}$ such that for
every arc $v_i \rightarrow v_j$, we have $i<j$.  The reader who is
interested in knowing more about DAGs is referred to reference
\cite{bjg} which supplies a wealth of information on DAGs.

Suppose that $D$ is an acyclic orientation of $G$. An arc 
of $D$, or its underlying edge, is called {\em dependent} (in $D$) if 
its reversal creates a directed cycle in the resulted orientation.
Note that $u \rightarrow v$ is a dependent arc if and only if
there exists a directed walk of length at least two from $u$ to $v$.
Let $d(D)$ denote the number of dependent arcs in $D$. Let
$d_{\min}(G)$ and $d_{\max}(G)$ be, respectively, the minimum
and the maximum values of $d(D)$ over all acyclic orientations
$D$ of $G$. It is known (\cite{fflw}) that $d_{\max}(G)=\|G\|-|G|+
c$ for a graph $G$ having $c$ connected components.

An interpolation question asks whether $G$ has an acyclic orientation
with exactly $d$ dependent arcs for every $d$ satisfying $d_{\min}(G)
\leqslant d \leqslant d_{\max}(G)$. Following West (\cite{west}),
we call $G$ {\em fully orientable} if its interpolation question has
an affirmative answer. Note that forests are trivially fully orientable.
It is also easy to see (\cite{ll}) that a graph is fully
orientable if all of its connected components are. West \cite{west}
showed that complete bipartite graphs are fully orientable. Let
$\chi(G)$ denote the {\em chromatic number} of $G$, i.e., the least
number of colors to color the vertices of $G$ so that adjacent
vertices receive 
different colors. Let $g(G)$ denote the {\em girth}
of $G$, i.e., the length of a shortest cycle of $G$ if there is any,
and $\infty$ if $G$ is a forest. Fisher, Fraughnaugh, Langley, and
West \cite{fflw} showed that $G$ is fully orientable if $\chi(G) < g(G)$.
They also proved that $d_{\min}(G)=0$ when $\chi(G) < g(G)$. In fact,
$d_{\min}(G)=0$ if and only if $G$ is a {\em cover graph}, i.e., the
underlying graph of the Hasse diagram of a partially ordered set.
(\cite{rt}, Fact 1.1).

A number of graph classes have been shown to consist of fully
orientable graphs in recent years. Here, we give a brief summary
of some
results. 

A graph is called {\em $2$-degenerate} if each of its subgraphs
contains a vertex of degree at most two. Lai, Chang, and Lih
\cite{lcl} have established the full orientability of 2-degenerate
graphs that generalizes a previous result for outerplanar graphs
(\cite{llt}).
A {\em Halin graph} is a plane graph obtained by drawing a tree
without vertices of degree two in the plane, and then drawing a
cycle through all leaves in the plane. A {\em subdivision of an edge}
of a graph is obtained by replacing that edge by a path consisting of
new internal vertices. A {\em subdivision} of a graph is obtained
through a sequence of subdivisions of edges. Lai and Lih  \cite{ll}
showed that subdivisions of Halin graphs and graphs with maximum
degree at most three are fully orientable.
In \cite{llt}, Lai, Lih, and Tong proved that a graph $G$ is fully
orientable if $d_{\min}(G) \leqslant 1$. This generalizes the results
in \cite{fflw} mentioned before.

The main purpose of this paper is to show that 
the class of fully orientable graphs includes
the important class of chordal graphs. 

Let $C$ be a cycle of a graph $G$. An edge $e$ of $G$ is called a
{\em chord} of $C$ if the two endpoints of $e$ are non-consecutive
vertices on $C$. A graph $G$ is called {\em chordal} if each cycle
in $G$ of length at least four possesses a chord. Chordal graphs are
variously known as {\em triangulated graphs} \cite{be}, {\em
rigid-circuit graphs} \cite{di}, and {\em monotone transitive graphs}
\cite{ro70} in the literature. Chordal graphs can be characterized in
a number of different ways. (For instance, \cite{bu}, \cite{di},
\cite{fg}, \cite{ga}, and \cite{ro70}).

Chordal graphs have applications in areas such as the solution of
sparse symmetric systems of linear equations \cite{ro72}, data-base
management systems \cite{ty}, knowledge based systems \cite{ebt},
and computer vision \cite{cm}. The importance of chordal graphs
primarily lies in the phenomenon that many NP-complete problems can
be solved by polynomial-time algorithms for chordal graphs.

We need the following characterization of chordal graphs to prove
our main result. 
A complete subgraph of a graph $G$ is called a {\em clique} of $G$. 
A vertex $v$ of a graph $G$ is said to be {\em simplicial} if
$v$ together with all its adjacent vertices 
induce 
a clique in $G$. An
ordering $v_1, v_2, \ldots , v_n$ of all the vertices of $G$ forms
a {\em perfect elimination ordering} of $G$ if each $v_i$, $1
\leqslant i \leqslant n$, is simplicial in the subgraph induced by
$v_i, v_{i+1}, \ldots , v_n$.

\begin{theorem}{\rm \cite{ro72}}
A graph $G$ is a chordal graph if and only if it has a perfect
elimination ordering.
\end{theorem}

The reader is referred to Golumbic's classic \cite{go} for more
information on chordal graphs.

%
%
\section{Results}
%
%

Up to the naming of vertices, any acyclic orientation $D$ of
$K_n$ produces the topological ordering $v_1,\ldots ,v_n$
such that the arc $v_i\rightarrow v_j$ belongs to $D$ if and
only if $i<j$. Moreover, $v_i\rightarrow v_j$ is a dependent
arc in $D$ if and only if $j-i>1$. A vertex is called a {\em source}
(or {\em sink}) if it has no ingoing (or outgoing) arc. 
The following observation is very useful in the sequel.
Let $D$ be an acyclic orientation of the complete graph
$K_n$ where $n \geqslant 3$.
The number of dependent arcs in $D$ incident to a vertex $v$
is $n-2$ if $v$ is the source or the sink of $D$ and is $n-3$
otherwise.

In this section, we assume that the clique $Q$ of a graph $G$ has $q$
vertices. Let $G'$ be the graph obtained from $G$ by adding a new
vertex $v$ adjacent to all vertices of $Q$. 
We see that $d_{\max}(G')=\|G'\|-|G'|+1=(\|G\|+q)-(|G|+1)+1
=(\|G\|-|G|+1)+q-1=d_{\max}(G)+q-1$.
Furthermore, we have the following.

\begin{lemma}\label{main}
{\rm ({\bf 1)}}\
If $G$ has an acyclic orientation $D$ with $d(D)=d$, then $G'$ has an
acyclic orientation $D'$ with $d(D')=d+q-1$.

{\rm ({\bf 2)}}\
We have $d_{\min}(G')$ equal to $d_{\min}(G)+q-2$ or $d_{\min}(G)+q-1$.
\end{lemma}

\begin{proof}
The statements hold trivially when $q=1$. Assume $q \geqslant 2$.

{\bf (1)}\
Let $D'$ be the extension of $D$ into $G'$ by making $v$ into a source.
Clearly, $D'$ is an acyclic orientation. Let $v_1,\ldots ,v_q$ be the
topological ordering of vertices of $Q$ with respect to $D$. Suppose
that $x \rightarrow y$ is a dependent arc in $D'$. 

Case 1.\
If this arc is in $D$,
then it is already dependent in $D$ since $v$ is a source in $D'$. 


Case 2.\
If this arc is $v \rightarrow v_1$, then, for some $2 \leqslant i 
\leqslant q$, a directed path $v \rightarrow v_i \rightarrow z_1 
\rightarrow \cdots  \rightarrow z_i  \rightarrow v_1$ of length at 
least three would be produced such that  $z_1, \ldots , z_i$ are all 
vertices in $G$. It follows that $v_1 \rightarrow v_i 
\rightarrow z_1 \rightarrow \cdots  \rightarrow z_i  \rightarrow v_1$ is a
directed cycle in $D$, contradicting to the acyclicity of $D$.

Case 3.\
If this arc is $v \rightarrow v_k$ for $2 \leqslant k\leqslant q$,
then it is a dependent arc in $D'$ since $v \rightarrow 
v_{k-1} \rightarrow v_k$ is a directed path of length two.

Therefore, $d(D')=d+q-1$.

{\bf (2)}\
By statement {\bf (1)}, we have $d_{\min}(G') \leqslant d_{\min}(G)
+q-1$. Let $D'$ be an acyclic orientation of $G'$ with $d(D')=
d_{\min}(G')$. 
Since the subgraph induced by $Q$ and $\{v\}$ is a clique of order $q+1$, 
the number of dependent arcs in $D'$ incident to $v$ is $q-1$ or $q-2$.
Let $D$ be the restriction of $D'$ to $V(G)$. Then we have $d_{\min}(G')
=d(D')\geqslant d(D)+q-2\geqslant d_{\min}(G)+q-2$.
\end{proof}

\bigskip

Since every number $d$ satisfying $d_{\min}(G) +q-1 \leqslant d\leqslant
d_{\max}(G)+q-1$ is achievable as $d(D')$ for some acyclic orientation
$D'$ of $G'$ by {\bf (1)}, the following is a consequence of {\bf (2)}.

\begin{corollary}
If $G$ is fully orientable, so is $G'$.
\end{corollary}

The above theorem amounts to preserving full orientability by the addition
of a simplicial vertex. Hence, by successively applying it to the reverse
of a perfect elimination ordering of a connected chordal graph, every
such graph is fully orientable. Our main result thus follows.

\begin{theorem}
If $G$ is a chordal graph, then $G$ is fully orientable.
\end{theorem}

{\bf Remark}.\
Adding a simplicial vertex may not increase the maximum
and the minimum numbers of dependent edges by the same amount. For
instance, any acyclic orientation of a triangle gives rise to exactly
one dependent arc. However, the graph $K_4$ minus an edge, which is
obtained from a triangle by adding a simplicial vertex, has minimum
value one and maximum value two. 

\bigskip

Now we want to give a characterization
to tell which case in {\bf (2)} of Lemma \ref{main} will happen.
A dependent arc in $Q$ is said to be {\em non-trivial} with 
respect to the acyclic orientation $D$ if it is dependent in $D$
but not in the induced orientation $D[Q]$.  Equivalently, any 
directed cycle obtained by reversing that arc contains vertices 
not in $Q$. 

\begin{lemma}\label{more}
Assume $q \geqslant 2$. There is an acyclic orientation $D$ of $G$
such that $Q$ has a dependent arc that is non-trivial with respect 
to $D$ if and only if $D$ can be extended to an acyclic orientation 
$D'$ of $G'$ with $d(D') = d(D)+q-2$.
\end{lemma}

\begin{proof}
$(\Rightarrow)$\
Assume that $D$ is an acyclic orientation of $G$ such that $Q$ has
a dependent arc that is non-trivial with respect to $D$. Let $v_1,
\ldots ,v_q$ be the topological ordering of the vertices of $Q$ with 
respect to $D$. The arcs in the set $\{v_i \rightarrow v_j \mid j - i 
> 1 \}$ are dependent arcs that are not non-trivial with respect to $D$.

%
%

By our assumption,
we can find $1 \leqslant k < q$ such that $v_k \rightarrow v_{k+1}$
is a dependent arc in $D$. We obtain an extension $D'$ of $D$ into
$G'$ by defining $v_a \rightarrow v$ for all $a\leqslant k$ and
$v\rightarrow v_b$ for all $b>k$. This $D'$ must be acyclic for
otherwise a directed path would be produced in $D$, contradicting
the acyclicity of $D$. The set $\{vv_r \mid r\neq k,k+1\}$ gives
rise to a set of dependent arcs in $D'$ and both $vv_k$ and
$vv_{k+1}$ are not dependent in $D'$. Moreover, an edge of $G$ is
dependent in $D$ if and only if it is dependent in $D'$. Therefore,
$d(D')=d(D)+q-2$.

$(\Leftarrow)$\
Assume that $D$ can be extended to an acyclic orientation $D'$ of
$G'$ with $d(D') = d(D)+q-2$.
If the vertex $v$ is a source or a sink, then
$d(D')=d(D)+q-1$, contradicting our assumption. Without loss of
generality, we may suppose that, for some $1 \leqslant k < q$, $v_a
\rightarrow v$ for all $1 \leqslant a \leqslant k$ and $v
\rightarrow v_{k+1}$. The acyclicity of $D'$ implies that $v
\rightarrow v_b$ for all $b > k$. Hence, the arc $v_k \rightarrow
v_{k+1}$ is dependent in $D'$
for $v_k \rightarrow v \rightarrow v_{k+1}$ is a directed path of length two.
Since the $q-2$ arcs $vv_r$  ($r\neq
k,k+1$) incident to $v$ are already dependent in $D'$, it forces
$v_k \rightarrow v_{k+1}$ to be a dependent arc in $D$. Therefore,
$v_k \rightarrow v_{k+1}$ is non-trivial with respect to $D$.
\end{proof}

\begin{corollary}\label{more'}
Assume $q \geqslant 2$. There is an acyclic orientation $D$ of $G$
such that $d(D)=d_{\min}(G)$ and $Q$ has 
a dependent arc that is non-trivial with respect to $D$
if and only if $d_{\min}(G') =d_{\min}(G)+q-2$.
\end{corollary}

{\bf  Remark}.\
For the complete graph $K_n$ on $n$ vertices, $d_{\min}(K_n)=
d_{\max}(K_n)=(n-1)(n-2)/2$ is a well-known fact (\cite{west}).
Hence, the condition in Theorem \ref{more} and Corollary \ref{more'}
that $Q$ has a dependent arc that is non-trivial with respect to $D$ 
can be replaced by the condition that $Q$ has more than $(q-1)(q-2)/2$ 
arcs that are dependent in $D$.

\bigskip

In contrast to the addition of a simplicial vertex, the deletion
of a simplicial vertex may destroy full orientability. The following
example attests to this possibility.

Let $K_{r(n)}$ denote the complete $r$-partite graph each of whose
partite sets has $n$ vertices. It is proved in \cite{clt} that
$K_{r(n)}$ is not fully orientable when $r \geqslant 3$ and $n
\geqslant 2$. Any acyclic orientation of $K_{3(2)}$ has 4, 6, or 7
dependent arcs. Figure \ref{fig1} shows an acyclic orientation of
$K_{3(2)}$ with 6 dependent arcs. Two dependent arcs appear in the
innermost triangle 146. Let $K'$ be the graph obtained from $K_{3(2)}$
by adding a vertex $v$ adjacent to vertices 1, 4, and 6. By Lemma
\ref{main}, there exist acyclic orientations of $K'$ with 6, 8, or
9 dependent arcs. Actually, $d_{\max}(K')=9$. Applying Lemma
\ref{more} to Figure \ref{fig1}, we obtain an acyclic orientation
of $K'$ with 7 dependent arcs. Any acyclic orientation of $K_{3(2)}$
with 4 dependent arcs cannot have two dependent arcs from the triangle
146 since there are three triangles each of which is edge-disjoint
from the triangle 146 and we know that every triangle must have one
dependent arc. It follows from
Corollary \ref{more'}
that $d_{\min}(K')=6$. Hence, $K'$ is fully orientable. The deletion
of the simplicial vertex $v$ from $K'$ produces $K_{3(2)}$ that is
not fully orientable.

\begin{figure}
\begin{center}
\includegraphics{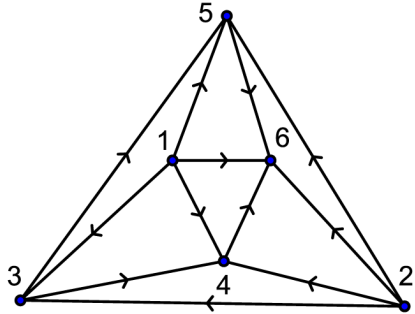}
\caption{An acyclic orientation of $K_{3(2)}$ with 6 dependent arcs.}
\label{fig1}
\end{center}
\end{figure}

{\bf Acknowlegment}.\ The authors are indebted to an anonymous referee
for constructive comments leading to great improvements.


\end{document}